\documentclass[12pt]{article}
\usepackage{xspace}
\usepackage[french,english]{babel}
\usepackage[latin1]{inputenc}
\usepackage[T1]{fontenc}
\usepackage{amsmath,amssymb,amsfonts,amsthm,amscd}
\usepackage{fullpage}
\usepackage{url}
\usepackage[hidelinks]{hyperref}
\usepackage{algorithm}
\usepackage{algorithmic}
\usepackage{tikz-cd}
\usepackage[fixlanguage]{babelbib}
\usepackage{xcolor}
\numberwithin{equation}{section}
\def\bildist#1#2{#1(#2)}
\def\paren#1{\left(#1\right)}
\def\epsilona#1#2{\epsilon_{#1,#2}}
\def\XX{\ensuremath{X}}
\def\id{\mathrm{Id}}
\def\ZZ{\ensuremath{\mathbb Z}}

\def\QQ{\ensuremath{\mathbb Q}}
\def\CC{\ensuremath{\mathbb C}}

\def\PP{\ensuremath{\mathbb P}}
\def\Qp{{\ensuremath{\QQ_p}}}

\def\Zp{\ensuremath{\ZZ_p}}
\def\bD{\ensuremath{\mathbf D}}

\def\cau{\mathcal{C}}
\def\caum#1{\mathcal{C}\paren{#1}}
\def\cC{\ensuremath{\mathcal C}}
\def\cD{\ensuremath{\mathcal D}}
\def\cL{\ensuremath{\mathcal L}}
\def\cH{\ensuremath{\mathcal H}}
\def\cV{\ensuremath{\mathcal V}}
\DeclareMathOperator{\Fil}{Fil}
\DeclareMathOperator{\Gr}{Gr}
\DeclareMathOperator{\Hom}{Hom}
\DeclareMathOperator{\Fonct}{Fonct}
\DeclareMathOperator{\ord}{ord}

\DeclareMathOperator{\Gl}{GL}
\def\limproj#1{\displaystyle{\lim_{\underset{#1}{
\leftarrow}}}\
}
\def\Per{\mathcal{P}er}
\def\Perp{\Per^{(p)}}
\def\Perpi{\Per_\infty^{(p)}}
\def\smallmat#1{\left(\begin{smallmatrix}#1\end{smallmatrix}\right)}
\def\norm#1{\lVert #1 \rVert}
\def\cU{\ensuremath{\mathcal U}}
\def\betab#1#2{\beta_{#1}^{(#2)}}
\def\charun{{\mathbf 1}}
\def\actk{\mid_k}
\def\act{\mid}
\theoremstyle{plain}
\newtheorem{thm}{Théorème}[section]

\newtheorem{lem}[thm]{Lemme}
\newtheorem{prop}[thm]{Proposition}
\newtheorem{cor}[thm]{Corollaire}
\theoremstyle{definition}

\newtheorem{rem}[thm]{Remarque}

\title{Symboles modulaires surconvergents et fonctions $L$ $p$-adiques}
\usepackage[noblocks]{authblk}

\author[*]{Karim Belabas}
\author[**]{Bernadette Perrin-Riou}
\affil[*]{
Univ. Bordeaux, CNRS, Bordeaux INP, IMB, UMR 5251, F-33400, Talence,
France;
INRIA, IMB, UMR 5251, F-33400, Talence, France}
\affil[**]{Université Paris-Saclay, CNRS, Laboratoire de mathématiques d'Orsay, 91405, Orsay, France.}

\begin{document}
\maketitle


\selectlanguage{english}
\begin{abstract}
We come back to the construction of $p$-adic $L$-functions attached to cusp
forms of even weight $k$ in the spirit of G. Stevens, R. Pollack \cite{PS}
and M. Greenberg \cite{mgreenberg} with a new unified presentation including
the non-ordinary case. This construction is based on Stevens's modular
symbols rather than $q$-developments. We review the proofs in order to obtain
an effective algorithm guaranteeing a given $p$-adic accuracy.
\end{abstract}
\selectlanguage{french}
\begin{abstract}
Nous reprenons la construction des fonctions $L$ $p$-adiques associées aux
formes paraboliques de poids $k$ pair dans l'esprit de G. Stevens, R.
Pollack \cite{PS} et M. Greenberg \cite{mgreenberg} avec une présentation
différente et unifiée dans le cas non ordinaire. Cette construction est
basée sur les symboles modulaires de Stevens plutôt que sur les
$q$-développements. Nous reprenons les démonstrations pour maîtriser les
approximations $p$-adiques et obtenir un algorithme effectif.
\end{abstract}

Soit $p$ un nombre premier.
Il y a de nombreuses constructions des fonctions $L$ $p$-adiques
associées à une forme parabolique $F$ pour un sous-groupe de congruence
$\Gamma$ de niveau $N$ et de poids $k$ (\cite{av},
  \cite{panchishkin02}, \cite{MTT}, \dots).
L'une d'elles, due à Pollack-Stevens~\cite{PS} et
Greenberg~\cite{mgreenberg}, part
du symbole modulaire de Stevens associé qui est un élément de
$\Hom_\Gamma(\Delta_0, V_k)$ où $\Delta_0$ est le sous-module de $\ZZ[\PP^1(\QQ)]$
formé des diviseurs de degré 0 et $V_k$ l'espace des polynômes homogènes
de degré $k-2$ pour des actions de $\Gamma$ convenables. En supposant que
$F$ est une forme propre pour l'opérateur de Hecke $T_p(N)$ et pour
un certain $\Qp$-espace vectoriel $D$ de dimension 1 ou 2 selon la valeur propre de $F$,
on associe à $F$ un \textsl{symbole modulaire} à valeurs dans $D \otimes_\QQ V_k$
qui est vecteur propre pour l'opérateur de Hecke $U_p$ de niveau $Np$.
Par itération de $U_p$, on construit alors un symbole modulaire
$\Phi_\infty$ à valeurs dans l'espace des distributions sur $\Zp$ et fixe par $U_p$.
La fonction $L$ $p$-adique est alors associée comme usuellement
à la distribution $\Phi_\infty((\infty,0))$ ou plutôt à sa restriction
à $\Zp^\times$.
On associe donc à $F$ sa transformée de Cauchy $\cC_F$ dans $\Qp[[w]]$ donnée par
$$\cC_F(w)=\int_{\Zp} \frac{1}{1-zw}
   d \left(\Phi_\infty((\infty,0))\right)(z)\;,$$
qui est donc formée à l'aide des moments de $\Phi_\infty((\infty,0))$.
L'intérêt de cette présentation est d'obtenir facilement les dérivées
successives de la fonction $L$ $p$-adique aux valeurs critiques.
Nous avons implanté cette construction dans le système
Pari/GP~\cite{pari} (fonctions \texttt{mspadicmoments} et
\texttt{mspadicL}; la fonction \texttt{ellpadicL} optimise le cas particulier
$k = 2$, permettant en particulier le calcul des invariants d'Iwasawa
(\texttt{ellpadiclambdamu}) et la vérification numérique de la conjecture
de Birch et Swinnerton-Dyer $p$-adique (\texttt{ellpadicbsd}).

Donnons le plan de l'article.
Le premier paragraphe est un rappel sur les distributions.
Dans le deuxième paragraphe, nous définissons les filtrations qui permettent
de maitriser la convergence des séries. Dans le troisième paragraphe, nous
démontrons l'existence sous certaines conditions d'un symbole modulaire à
valeurs dans les distributions sur $\Zp$ à partir d'un symbole modulaire
à valeurs dans les polynômes de degré inférieur ou égal à $k-1$
(proposition \ref{prop:phioo}). Dans le quatrième paragraphe, nous donnons les trois situations
venant des formes modulaires (cas ordinaire, cas semi-stable, cas
supersingulier) où notre construction s'applique,
puis nous relions les valeurs des fonctions $L$ $p$-adiques ainsi définies
aux valeurs complexes traditionnelles.


\section{Distributions}
Soit $\Gamma$ un sous-groupe de congruence de niveau $N$ et
$\Gamma_0=\Gamma \cap \Gamma_0(p)$.
Soit $$\Sigma_0(p)= \{\gamma \in M_2(\ZZ) \text{ tel que }
\det(\gamma) \neq 0, p \mid c, p \nmid a\}.$$
Soit $A = \Qp[[z]]$ l'anneau des séries entières $\sum_{n=0}^\infty a_n z^n$
à coefficients dans $\Qp$ telles que $a_n\to 0$.
Pour $k \in \Zp$, on munit $A$ d'une action à gauche de
$\Sigma_0(p)$:
\def\actdet#1#2{}
$$ \gamma \cdot_{k} f (z) = {\actdet{\gamma}{-k+1}}(a+cz)^{k-2}
f \paren{\frac{b+d z}{a+c z}} $$
et on le note alors $A_k$.
Si $k$ est un entier $\geq 2$,
on note $\QQ[z]_{k-2}$ le sous-espace vectoriel de $A_k$ des polynômes de
degré inférieur ou égal à $k-2$.
Il est stable par l'action de
$\Sigma_0(p)$.
Soit $\cD_k$ le dual continu de $A_k$;
pour $\mu\in \cD_k$ et $f\in A_k$, on note indifféremment $\mu(f)
=\int f\,d\mu$.
On munit $\cD_k$ de l'action duale de
$\Sigma_0(p)$:
$$ \int f \,d\mu\actk \gamma = \int \gamma \cdot_{k}f \,d\mu.
$$
Pour $k$ entier supérieur ou égal à $2$, définissons
$\rho_k: \cD_k \to \Qp[\XX]_{k-2}$
par
\begin{equation*}
\rho_k(\mu) =\int (1-\XX z )^{k-2} d \mu(z).
\end{equation*}
On a
\begin{equation*}
\begin{split}
\rho_k(\mu\actk \gamma)(\XX)&= \int \big(a+cz-\XX (b+dz)\big)^{k-2} d \mu(z)
\\&=
(a-b\XX)^{k-2}\int \paren{1- \frac{-c + d\XX}{a-b\XX}z}^{k-2} d \mu(z)
\\
&=
(a-b\XX)^{k-2}\rho_k(\mu)\paren{\frac{-c + d\XX}{a-b\XX}}.
\end{split}
\end{equation*}
D'où
$$
\rho_k(\mu\actk \gamma)= {^t}\overline{\gamma} \cdot_k \rho_k(\mu)
= \rho_k(\mu)\actk {^t\gamma^*}\,.
$$
\begin{prop}
On a le diagramme commutatif de $\Sigma_0(p)$-modules
\begin{equation*}
\begin{matrix}
\cD_k & \times & A_k & \to& \Qp\\
\rho_k\downarrow&&\uparrow &&||\\
\Qp[\XX]_{k-2} & \times & \Qp[z]_{k-2} & \overset{\left\langle\cdot,\cdot\right\rangle}{\to}& \Qp\\
\end{matrix}
\end{equation*}
où la forme bilinéaire $\Qp[\XX]_{k-2} \times \Qp[z]_{k-2} \to \Qp$
est donnée par
$$ \paren{\sum_i \mu_i \XX^i, \sum_i f_i z^i} \mapsto
\sum_{j=0}^{k-2} (-1)^j\frac{\mu_jf_j}{\binom{k-2}{j}}\,.$$
\end{prop}
\begin{proof}
Calcul explicite à partir de la formule
\begin{equation*}
\begin{split}
\rho_k(\mu)&=\sum_{j=0}^{k-2} (-1)^j \binom{k-2}{j}\mu(z^j) \XX^j\,.
\end{split}
\end{equation*}
\end{proof}
Soit $\cau: \cD \to \Qp[[w]]$ la transformation de Cauchy formelle
 donnée par la série formelle
\begin{equation*}
 \caum{\mu}=\sum_{j=0}^\infty \mu(z^j) w^j = \mu\paren{\frac{1}{1-zw}}.
\end{equation*}
Soit $\cH$ le sous-anneau de $\Qp[[w]]$ formé
des fonctions analytiques sur tout disque $B(0,\rho)$ pour
$\rho < 1$, c'est-à-dire dont les coefficients vérifient
$\sup_n |a_n| \rho^n < \infty$ pour tout $\rho<1$.
Définissons les sous-modules de $\cD$ munis de l'action
induite de $\cD_k$

\begin{equation*}
\bD_k(\Zp) = \cau^{-1}(\Zp[[w]]) \subset
\bD_k= \cau^{-1}(\Qp\otimes \Zp[[w]])
\subset
\bD^{\dag}_k= \cau^{-1}(\cH).
\end{equation*}
\begin{lem} Les sous-modules $\bD_{k}^{\dag}$ et $\bD_k$ de $\cD_k$ sont stables
par $\Sigma_0(p)$.
\end{lem}
\begin{proof}
Si $\gamma=\smallmat{a&b\\c&d} \in \Sigma_0(p)$ et $\mu\in \cD$,
on a
\begin{equation}
\label{formule}
\begin{split}
{\actdet{\gamma}{k-1}}\int z^n d (\mu\actk \gamma)(z)
&=\int (a+cz)^{k-2-n} (b+dz)^n d \mu(z)\\
&=
a^{k-2-n} \sum_{l=0}^n \sum_{j=0}^\infty \binom{k-2-n}{j} (c/a)^{j}
\binom{n}{l} b^{n-l} d^l \mu(z^{j+l})
\end{split}
\end{equation}
et donc
\begin{equation*}
\caum{\mu\actk \gamma}(w) =
  {\actdet{\gamma}{-k+1}}\sum_{n=0}^\infty a^{k-2-n}
  \paren{\sum_{l=0}^n \sum_{j=0}^\infty\binom{k-2-n}{j} (c/a)^{j}
  \binom{n}{l} b^{n-l} d^l \caum{\mu}_{j+l}} w^n.
\end{equation*}
Remarquons que les coefficients binomiaux
$\binom{l}{j}$ pour $l\in \ZZ$ et $j\geq 0$ appartiennent à $\Zp$.
Si $\gamma \in \Sigma_0(p)$,
$p$ divise $c$ et $a$ est une unité en $p$.

Supposons que $\mu \in \bD_k^\dag$.
On déduit facilement de la formule que
$$\norm{\caum{\mu\actk\gamma}}_\rho
  \leq \sup_{j\geq 0} \paren{\frac{|c|_p}{\rho}}^j \norm{\caum{\mu}}_\rho
  \leq \norm{\caum{\mu}}_\rho
$$
pour $\rho$ vérifiant $|c|_p\leq \rho< 1$, donc pour tout $\rho <1$.
Ce qui montre que
$\mu\actk \gamma$ appartient à $\bD_k^\dag$. Si $\mu\in \bD_k$, le
même calcul montre que
$$\norm{\caum{\mu\actk\gamma}}_\rho \leq \norm{\caum{\mu}}_1.$$
Donc $\mu\actk \gamma$ appartient à $\bD_k$. On en déduit le lemme.
\end{proof}
Lorsque $c=0$, on a simplement
\begin{equation*}
\caum{\mu\actk \begin{pmatrix}a&b\\0&d\end{pmatrix}}(w) =
{\actdet{\gamma}{1-k}} a^{k-1} \frac{1}{a-bw}
   \caum{\mu}\paren{\frac{d w}{a-b w}}.
\end{equation*}
Posons $\gamma_1=\smallmat{1&1\\0&1}$.
\begin{lem}
L'action de $\gamma_1$ sur $\cD$ ne dépend pas de $k$.
\begin{enumerate}
\item
Il n'existe pas de distribution non nulle $\mu$ tel que
$\mu\actk \gamma_1 = \mu$.
\item
On a
$$\ord_w(\caum{\mu}) <\ord_w(\caum{\mu\act(\gamma_1-1)}).$$
et
$\cD\actk (\gamma_1-1)$
est formé des distributions $\mu$ telles que $\int d\mu=0$,
i.e. telles que le coefficient constant de $\caum{\mu}$ est nul.
\end{enumerate}
\end{lem}
\begin{proof}
On a pour tout entier $k\in \ZZ$
\begin{equation*}
\caum{\mu\actk \gamma_1}(w)
=\frac{1}{1-w} \caum{\mu}\paren{\frac{w}{1 - w}}.
\end{equation*}
Soit $n\geq 0$ et $f(w)$ un élément de $\Qp[[w]]$
de la forme $w^n + \alpha w^{n+1} +O(w^{n+2})$ tel que
$f(w)=\frac{1}{1-w} f(\frac{w }{1-w})$.
On a
\begin{equation*}
\begin{split}
\frac{1}{1-w} f\paren{\frac{w}{1-w}} &=
 w^n (1+w)^{n+1} + \alpha w^{n+1} (1+w)^{n+1} +O(w^{n+2})\\
 &=
w^n +(\alpha+n+1)w^{n+1} +O(w^{n+2}).
\end{split}
\end{equation*}
L'équation $f(w)=\frac{1}{1-w} f\paren{\frac{w }{1 - w}}$ implique que
$$
w^n + \alpha w^{n+1} =w^n +(\alpha+n+1)w^{n+1}
$$
ce qui est impossible. D'où la première assertion. On a
\begin{equation*}
\begin{split}
\caum{\mu\actk (\gamma_1-1)} &=
\sum_{n=0}^\infty\left (
\sum_{l=0}^{n-1} \binom{n}{l}\int z^{l}d \mu \right ) w^n.
\end{split}
\end{equation*}
Son terme constant est nul. Plus généralement, on obtient
$$\ord_w(\caum{\mu}) <\ord_w(\caum{\mu\act(\gamma_1-1)}).$$
Finalement, l'équation $\caum{\mu\actk (\gamma_1-1)} = \sum_{n > 0} a_n w^n$
est équivalente à un système triangulaire :
\begin{equation*}
\begin{cases}
x_0 &= a_1\\
x_0 + 2x_1 &= a_2\\
\cdots\\
x_0 + n x_1 + \cdots + n x_{n-1}&=a_n\\
\cdots\\
\end{cases}
\end{equation*}
et a donc une solution, ce qui termine la démonstration du lemme.
\end{proof}
On note $\cU_p$ l'opérateur de $\bD_k^{\dag}$ défini par
$\cU_p(\mu)=\sum_{b=0}^{p-1} \mu\actk \betab{b}{p}$
avec $\betab{b}{p}=\smallmat{1&b\\0&p}$.
\begin{lem}
\label{up}
Si $\mu \in \bD_k^{\dag}$, alors
$\cU_p(\mu)$ appartient à $\bD_k$ et
$$ \norm{ \caum{\cU_p(\mu)} }_1 \leq \norm{ \caum{\mu} }_{p^{-1}}.$$
Si $\mu \in \bD_k(\Zp)$, alors
$\cU_p(\mu) $ appartient à $\bD_k(\Zp)$.
\end{lem}
\begin{proof} On a
\begin{equation}
\label{cauchybeta}
\begin{split}
\caum{\mu\actk \betab{b}{p}}(w)
&=\frac{1}{1-b w} \caum{\mu}\paren{\frac{p w}{1-b w}}
=\sum_{j=0}^\infty \mu(z^j)\frac{p^j w^j}{(1-b w)^{j+1}}\\
&=\sum_{j=0}^\infty p^j\mu(z^j)\sum_{i=0}^\infty \binom{-j-1}{i}(-b)^i w^{i+j}\\
&=
\sum_{n=0}^{\infty}\left (\sum_{j=0}^n
p^j\mu(z^j)\binom{-j-1}{n-j} (-b)^{n-j} \right)w^n.
\end{split}
\end{equation}
Si $\mu \in \bD_k^{\dag}$, alors $|p^j\mu(z^j)|_p < \norm{\caum{\mu}}_{p^{-1}}$.
On en déduit que
$\norm{\caum{\cU_p(\mu)} }_1 \leq \norm{\caum{\mu}}_{p^{-1}}$
et que $\cU_p(\mu)$ appartient à $\bD_k$.
\end{proof}
\begin{lem}\label{noyau}
Le noyau de $\rho_k$ est stable par $\cU_p$.
Si $\mu \in \bD_k(\Zp)$ est dans le noyau de $\rho_k$,
alors
$\cU_p(\mu)$ appartient à $p^{k-1}\bD_k(\Zp)$.
\end{lem}
\begin{proof}
On déduit la première assertion du fait que
le noyau de $\rho_k$ est formé des distributions $\mu$ tels que
$\caum{\mu}$ appartient à $w^{k-1} \Qp[[w]]$ et que
les opérateurs $\betab{a}{p}$ et $\cU_p$ stabilisent $w^n\Qp[[w]]$
pour tout entier $n$.
La formule \eqref{cauchybeta} implique précisément que si $\caum{\mu}\in
w^n\Zp[[w]]$, alors
$\caum{\mu\actk \betab{a}{p}} \in p^nw^n\Zp[[w]]$.
Ce qui montre la deuxième assertion en prenant $n=k-1$.
\end{proof}
Le lemme suivant ne nous sera pas utile. Nous le donnons pour être complet.
\begin{lem}Si $\ell$ est un nombre premier différent de $p$,
$$
\caum{\mu\actk T_\ell}=\sum_{a=0}^{\ell-1}
   \frac{1}{1-a w}\caum{\mu}\paren{\frac{\ell w}{1-a w}}
   + \ell^{k-2} \caum{\mu}\paren{\frac{w}{\ell}}.
$$
\end{lem}
\section{Modules d'approximation}
La filtration de $\bD_k^\dag$ induite par $\cau^{-1}(w^j\Qp[[w]])$ n'est pas stable par l'action
de $\Gamma_0$. Introduisons comme dans \cite{mgreenberg}
\footnote{Dans \cite{PS}, une filtration légèrement différente est utilisée.}
les filtrations $(\Fil^M)_{M\geq 0}$ sur $\bD_k^\dag$ et $\bD_k(\Zp)$
définies par
\begin{equation*}
\begin{split}
\Fil^0 \bD_k^\dag =& \{ \mu \in \bD_k^\dag \text{ tel que } \caum{\mu} \in w^{k-1} \Qp[[w]]\},\\
\Fil^M \bD_k^\dag=& \{ \mu \in \Fil^0 \bD_k^\dag\text{ tel que }\mu(z^{k-2+j}) \in p^{M-j+1} \Zp \text{ pour } j=1, \cdots, M\}\\
=& \{ \mu \in \bD_k^\dag\text{ tel que }\caum{\mu} \in p^{M} w^{k-1} \Zp[w/p] + w^{M+k-1}\Qp[[w]]\},\\
\Fil^0 \bD_k(\Zp) =&\{ \mu \in \bD_k^\dag \text{ tel que } \caum{\mu} \in w^{k-1} \Zp[[w]]\}\subset \bD_k(\Zp),\\
\Fil^M \bD_k(\Zp) =& \{ \mu \in \Fil^0 \bD_k(\Zp)\text{ tel que }\mu(z^{k-2+j}) \in p^{M-j+1} \Zp \text{ pour } j=1, \cdots, M\},\\
=& \{ \mu \in \bD_k(\Zp)\text{ tel que }\caum{\mu} \in p^{M} w^{k-1} \Zp[w/p] + w^{M+k-1}\Zp[[w]].
\end{split}
\end{equation*}
Soit $\Gr^M \bD_k^\dag$ la graduation associée.
En particulier, $\Gr^0 \bD_k^\dag$ est isomorphe à $\Qp[w]_{k-2}$.
Les quotients $\Gr^M \bD_k = \bD_k(\Zp)/\Fil^M\bD_k(\Zp)=\bD_k^\dag/\Fil^M\bD_k^\dag$
pour $M\geq 0$
sont des $\Zp$-modules de type fini
et l'application naturelle
$$\mu \mapsto \sum_{j=0}^\infty \mu(z^j) w^j \mapsto (\mu(z^j))_{0\leq j \leq M-1}$$
induit un isomorphisme
$$\Gr^M \bD_k \to \prod_{j=0}^{k-2} \Zp \times \prod_{j=k-1}^{M-1} \ZZ/p^{M+1-j}\ZZ.$$
De plus, pour tout entier $s\geq -M$, on a l'inclusion
\begin{equation}
\label{subset}
p^{s} \Fil^{M} \bD_k^\dag \subset \Fil^{M+s} \bD_k^\dag
\end{equation}
qui induit une application (multiplication par $p^s$)
\begin{equation*}
\Gr^{M} \bD_k \overset{p^{s}}{\to} \Gr^{M+s} \bD_k.
\end{equation*}
\begin{lem}
Les ensembles $\Fil^M \bD_k(\Zp)$ sont stables par l'action de $\Gamma_0$.
\end{lem}
\begin{proof}Voir \cite[Lemma 2]{mgreenberg}.
C'est une conséquence de la formule \eqref{formule}
appliquée aux deux matrices
$\smallmat{1&0\\c&1}$ et
$\smallmat{a&b\\0&d}$ puisque
$$\begin{pmatrix}a&b\\c&d\end{pmatrix}=
\begin{pmatrix}1&0\\c/a&1\end{pmatrix}\begin{pmatrix}a&b\\0&(ad-bc)/a\end{pmatrix}
.$$
\end{proof}
\begin{lem}\label{crucial0}
Soit $\mu \in \Fil^M \bD_k(\Zp)$.
Si $\beta=\smallmat{1&b\\0&p^s}$
avec $b\in \ZZ_p$ et $s\geq 1$, on a
$$\mu\actk \beta \in p^{s(k-1)-t}\Fil^{M+t}\bD_k(\Zp)$$
pour $t \leq s(k-1)$.
En particulier, pour $k \geq 2$ et $s = t = 1$, on a
$$\mu\actk \beta \in p^{k-2}\Fil^{M+1}\bD_k(\Zp).$$
\end{lem}
\begin{proof}
Soit $g$ l'image par $\cau$ d'un élément $\mu$ de $\Fil^M \bD_k(\Zp)$.
Comme
$$g \in p^{M}w^{k-1} \Zp[w/p] + w^{M+k-1}\Zp[[w]],$$
on a, pour $s\geq 1$ et pour $h\in \Zp^\times+w\Zp[[w]]$,
\begin{equation*}
\begin{split}
g(p^s w h(w))
&\in
p^{M+s(k-1)}w^{k-1} \Zp[w] + p^{s(M+k-1)} w^{M+k-1}\Zp[[w]]
\\ &\quad\quad\subset p^{M+s(k-1)} w^{k-1}\Zp[[w]].
\end{split}
\end{equation*}
Comme $\caum{\mu\actk \beta}(w) = \frac{1}{1 -b w} \caum{\mu}(p^s\frac{ w}{1-bw})$,
on applique la formule précédente à $h(w)=p^s \frac{w}{1-bw}$,
d'où
\begin{equation*}
\begin{split}
\mu\actk \beta \in p^{M+s(k-1)}\Fil^{0} \bD_k
\subset p^{s_1}\Fil^{M+s_2} \bD_k
\end{split}
\end{equation*}
pour tout couple d'entiers positifs $(s_1,s_2)$ tel que $s_1 + s_2 =s(k-1)$.
\end{proof}
\begin{cor}
\label{crucial}
Soit $\mu \in \Fil^M \bD_k(\Zp)$ et $s\geq 1$.
Pour tout entier $t$ inférieur ou égal à $s(k-1)$, on a
$$\cU_p^s(\mu) \in p^{s(k-1)-t}\Fil^{M+t}\bD_k(\Zp).$$
En particulier, si $s(k-1) \geq 2$,
$$\cU_p^s(\mu) \in p\Fil^{M+1}\bD_k(\Zp).$$
\end{cor}

\section{Symboles modulaires à valeurs dans les distributions}
Soit $\Delta_0$ le sous-module de $\ZZ[\PP^1(\QQ)]$
formé des diviseurs de degré 0.
L'opérateur $\cU_p$ induit un endomorphisme de
$\Hom_{\Gamma_0}(\Delta_0,\bD_k)$.
Comme $\Delta_0$ est stable par $\Gl_2(\QQ)$ et
que $\Delta_0$ est de type fini comme $\ZZ[\Gamma_0]$-module,
le corollaire \ref{crucial} s'étend à
$\Hom_{\Gamma_0}(\Delta_0,\Fil^M \bD_k)$.

\begin{prop}
L'image de $\Hom_{\Gamma_0}(\Delta_0,\bD_k^{\dag})$ par $\cU_p$
est contenue dans $\Hom_{\Gamma_0}(\Delta_0,\bD_k(\Zp))$.
\end{prop}
\begin{proof}
Se déduit du lemme \ref{up}.
\end{proof}
\begin{prop}[Pollack-Stevens \cite{PS}]
\label{rhosurjectif}
L'application $\pi_{0,*}$
$$\pi_{0,*}: \Hom_{\Gamma_0}(\Delta_0, \bD_k^{\dag})
 \to \Hom_{\Gamma_0}(\Delta_0, \Gr^0\bD_k^{\dag})$$
 est surjective.
\end{prop}
\begin{proof}\cite[\S 4]{PS}. Donnons ici la démonstration dans le cas où
$\gamma_1=\smallmat{1&1\\0&1}$ appartient à $\Gamma_0$.
Rappelons alors la structure de $\Delta_0$ en tant que $\Gamma_0$-module.
Il existe $a_1=(\infty,0)$, $\cdots$, $a_t=(r_t,s_t)$ dans $\Delta_0$, une involution
$*$ sur l'ensemble $\cV$ des $a_i$
et des éléments
$\gamma_{a_1}=\smallmat{1&1\\0&1}$, $\gamma_{a_2},\cdots,
\gamma_{a_t}$ de $\Gamma_0$ tels que
le $\Gamma_0$-module $\Delta_0$ soit engendré par les $a_1$, \dots, $a_t$
avec les relations
\begin{equation*}
\begin{cases}
&\sum_i a_i=0\\
&a_i + \gamma_{a_i} a_i^*=0 \text{ si $\gamma_{a_i}$ n'est pas d'ordre 3}\\
&a_i + \gamma_{a_i} a_i + \gamma_{a_i}^2 a_i=0 \text{ si $\gamma_{a_i}$ est d'ordre 3}
\end{cases}
\end{equation*}
Soit $\Psi \in \Hom_{\Gamma_0}(\Delta_0, \Gr^0\bD_k^{\dag})$.
Notons $\mu_{a_i}$ un relèvement de $\Psi(a_i)$ dans $\bD_k$
pour un système de représentants de $\cV$ modulo l'involution $*$.
Si $\gamma_{a_i}$ est elliptique d'ordre 3,
$\beta=(1+\gamma_{a_i}+\gamma_{a_i}^2)\mu_{a_i}$ appartient
à $w^{k-1}\bD_k^\dag$ et vérifie $(\gamma_{a_i}-1)\beta=0$.
On déduit de la relation $3=x^2+x+1 -(x+2)(x-1)$ que
$\beta= \frac{1}{3}(1+\gamma_{a_i}+\gamma_{a_i}^2)\beta$.
On peut donc changer $\mu_{a_i}$ de manière à ce que
$(1+\gamma_{a_i}+\gamma_{a_i}^2)\mu_{a_i}=0$. On peut de même
supposer que si $\gamma_{a_i}$ est elliptique d'ordre 2
(on a alors $a_i=a_i^*$), $(1+\gamma_{a_i})\mu_{a_i}=0$.
Passons à la première relation.
Puisque $\Psi \in \Hom_{\Gamma_0}(\Delta_0, \Gr^0\bD_k^{\dag})$, on a
$$ \mu=\sum_{j=1}^t \mu_{a_j}\actk (\gamma_{a_j}-1) \in \Fil^0 \bD_k^{\dag}.$$
Choisissons $i_0$ tel que $\gamma_{a_{i_0}}=\begin{pmatrix}a_{0}&b_{0}\\c_{0}&d_{0}\end{pmatrix}$,
avec $a_0$ différent de $\pm1$ et $c_{0}$ non nul
(il existe car les $\gamma_{a_i}$ engendrent $\Gamma_0$; notons que $i_0\neq 1$).
Nécessairement, $c_0$ est donc de valuation $p$-adique strictement positive.
Si $\mu^{(k-1)}$ est la distribution dont la transformée
de Cauchy est $w^{k-1}$, on a
$$\caum{\mu^{(k-1)} \actk (\gamma_{i_0}-1)}
= \sum_{n \geq 0} h_n w^n,$$
où
\begin{equation*}
h_n = \int (a_0+c_0z)^{k-2-n} (b_0+d_0z)^n \,d\mu^{(k-1)}.
\end{equation*}
Pour $n\leq k-2$, $h_n$ est nul; on a
\begin{equation*}
\begin{split}
h_{k-1}&=-1 + a_0^{-1}\sum_{j=0}^{k-1} (-1)^j (c_0/a_0)^j
\binom{k-1}{j}b_0^{j} d_0^{k-1-j}
\\
&=
-1 + a_0^{k-2} (a_0d_0-b_0c_0)^{k-1}=-1 + a_0^{k-2} \neq 0.
\end{split}
\end{equation*}
Donc
$\caum{\mu} -\caum{\mu_1\actk (\gamma_{i_0}-1)}
=O(w^{k})$
avec $\mu_1=\frac{\mu(z^{k-1})}{h_{k-1}}\caum{\mu^{(k-1)}}$
et il existe $\mu_2\in \Fil^0\bD^\dag_k$ tel que
$\mu = \mu_1\actk (\gamma_{i_0}-1) + \mu_2 \actk (\gamma_1-1)$.
On remplace
$\mu_{a_1}$ par $\mu_{a_1} - \mu_1$ et $\mu_{a_{i_0}}$ par $\mu_{a_{i_0}} - \mu_2$
sans changer leur projection dans $\Hom_{\Gamma_0}(\Delta_0, \Gr^0\bD_k^\dag)$
et on a alors la relation
$$\sum_{j=1}^t \mu_{a_j}\actk (\gamma_{a_j}-1)=0,$$
ce qui permet de définir un élément de
$\Hom_{\Gamma_0}(\Delta_0, \bD_k^\dag)$ dont l'image est $\Psi$.
\end{proof}
\begin{prop}
L'image par $\rho_{k,*}$ de $\Hom_{\Gamma_0}(\Delta_0, \bD_k)$
contient $\Hom_{\Gamma}(\Delta_0, \Qp[\XX]_{k-2})$.
\end{prop}
\begin{proof}
Soit $\Phi \in \Hom_{\Gamma}(\Delta_0, \Qp[\XX]_{k-2})$.
Montrons que $\Phi$ appartient à l'image de $\cU_p$. On a
$$\Phi\actk \begin{pmatrix}p&0\\0&1\end{pmatrix}\betab{b}{p} =
\Phi \actk \begin{pmatrix}1&b\\0&1\end{pmatrix}\begin{pmatrix}p&0\\0&p\end{pmatrix}
=\Phi\actk \begin{pmatrix}p&0\\0&p\end{pmatrix}.
$$
L'action de $\smallmat{p&0\\0&p}$ est triviale sur $\Delta_0$ et est la multiplication
par $p^{k-2}$ sur $\Qp[\XX]_{k-2}$.
Donc $$\Phi'=p^{-(k-1)}\Phi \actk \begin{pmatrix}p&0\\0&1\end{pmatrix}$$ appartient à
$\Hom_{\Gamma_0}(\Delta_0, \Qp[\XX]_{k-2})$ et vérifie $\cU_p(\Phi')= \Phi$.
Il existe $\mu'\in \Hom_{\Gamma_0}(\Delta_0, \bD_k^\dag)$
tel que $\rho_{k,*}(\mu')=\Phi'$ par la proposition \ref{rhosurjectif} et
$\rho_{k,*}(\cU_p(\mu'))=\Phi$ avec
$\cU_p(\mu') \in \Hom_{\Gamma_0}(\Delta_0, \bD_k)$.
\end{proof}

Pour $M>M'$,
les projections $ \pi_{M,M'} : \Gr^M \bD_k(\Zp) \to \Gr^{M'} \bD_k(\Zp)$
induisent des opérateurs
$$\pi_{M,M',*}: \Hom_{\Gamma_0}( \Delta_0,\Gr^M \bD_k(\Zp))
\to \Hom_{\Gamma_0}( \Delta_0,\Gr^{M'} \bD_k(\Zp))$$
compatibles avec les actions de $\Sigma_0$ et de $\cU_p$.

\begin{lem}
L'application naturelle
$$\Hom_{\Gamma_0}(\Delta_0, \bD_k(\Zp))\to \limproj{M}\Hom_{\Gamma_0}( \Delta_0,\Gr^M \bD_k(\Zp))$$
est un isomorphisme.
\end{lem}
Comme M. Greenberg \cite{mgreenberg}, on étend l'opérateur $\cU_p$
en un endomorphisme
de $\Fonct(\Delta_0, V)$ par
$$\cU_p(\Phi) =\sum_{b=0}^{p-1} \Phi\actk \betab{b}{p}$$
pour $\Phi$ fonction de $\Delta_0$ dans $V$ avec $V=\bD_k(\Qp)$ ou
$V=\Gr^M \bD_k(\Qp)$.
Cette extension dépend des matrices $\betab{b}{p}$
représentant les doubles classes de
$\Gamma_0 \smallmat{1&0\\0&p}\Gamma_0$.
Les propriétés d'intégralité montrées précédemment pour les
$\betab{b}{p}$ (et donc pour cette extension de $\cU_p$) restent vraies.

Soit $D$ un $\Qp$-espace vectoriel muni d'un automorphisme $\varphi$.
On suppose qu'il existe un réseau $L$ de $D$
stable par $\varphi^{-1}$, un entier $h$ et un réel positif $\lambda$ tels que
$$\varphi^{h} L \subset p^{-\lambda} L.$$
On a donc $p^{\lambda} L \subset \varphi^{-h}L$.

\begin{prop}
\label{prop:phioo}
Supposons que $\lambda < h(k-1)$.
Soit $\Phi$ un élément de $\Hom_{\Gamma_0}(\Delta_0,\Qp[\XX]_{k-2}\otimes D)$
tel que $(\cU_p\otimes \varphi)\Phi= \Phi$.
Il existe un élément $\Phi_\infty$ de
$\Hom_{\Gamma_0}(\Delta_0,\bD_k(\Zp)\otimes D)$ qui vérifie
$\rho_{k,*} \Phi_\infty = \Phi$ et
$(\cU_p\otimes \varphi)\Phi_\infty= \Phi_\infty$.
\end{prop}
Commençons la démonstration.
Quitte à multiplier $\Phi$ par une constante, on peut relever
$\Phi$ par~$\rho_{k,*}$
en un élément $\Phi_0$ de $\Hom(\Delta_0,\bD_k(\Zp)\otimes L)$
tel que $\caum{\Phi_{0}(\delta)}$ appartienne à
$p^{\lambda}\Zp[w]_{k-2}\otimes L$ pour tout $\delta\in\Delta_0$, et tel que
$$(\cU_p\otimes \varphi)\Phi_0 \in \Fonct(\Delta_0,\bD_k(\Zp)\otimes L).$$
L'image $\overline{\Phi}_0$ de $(\cU_p\otimes \varphi)\Phi_0$ dans
$\Fonct(\Delta_0,\Gr^0 \bD_k(\Zp)\otimes L)$ est un relèvement de $\Phi$
par $\rho_{k,*}$.
On relève $\overline{\Phi}_0$ en un élément $\widehat{\Phi}_0$
de $\Fonct(\Delta_0,\ZZ[w]_{k-2}\otimes L)$.
Il vérifie la congruence
\begin{equation*}
 \cU_p (\widehat{\Phi}_0) \equiv (1\otimes \varphi^{-1})\widehat{\Phi}_0\bmod
\Fonct(\Delta_0,\Fil^0\bD_k(\Zp)\otimes L).
\end{equation*}
\begin{lem} Supposons trouvé $\overline{\Phi}_{n-1}$ dans
$$\Fonct(\Delta_0,\bD_k(\Zp) \,/ \, p^{h(k-2)(n-1)}\Fil^{h(n-1)}\bD_k(\Zp)\otimes L)$$
tel que
\begin{equation*}
\cU_p(\overline{\Phi}_{n-1})\equiv
 (1\otimes \varphi^{-1}) \overline{\Phi}_{n-1}
  \bmod p^{h(k-2)(n-1)}\Fonct(\Delta_0,\Fil^{h(n-1)}\bD_k(\Zp)\otimes L).
\end{equation*}
Si $\widehat{\Phi}_{n-1}$ est un relèvement dans $\Fonct(\Delta_0,
\bD_k(\Zp)\otimes L)$ de $\overline{\Phi}_{n-1}$, alors
$\overline{\Phi}_{n} = \cU_p^h(\widehat{\Phi}_{n-1})$ vérifie
\begin{equation}
\begin{cases}
  \overline{\Phi}_{n} &\equiv (1\otimes \varphi)^{-h} \overline{\Phi}_{n-1}
 \bmod p^{h(k-2)(n-1)}\Fonct(\Delta_0,\Fil^{h(n-1)}\bD_k(\Zp)\otimes L)
\\
 \cU_p( \overline{\Phi}_{n}) &\equiv
 (1\otimes \varphi)^{-1}\overline{\Phi}_{n} \bmod p^{h(k-2)n}\Fonct(\Delta_0,\Fil^{hn}\bD_k(\Zp)\otimes L)
.
\end{cases}
\end{equation}
Il ne dépend pas du relèvement choisi modulo
$$p^{h(k-2)n}\Fonct(\Delta_0,\Fil^{hn}\bD_k(\Zp)\otimes L)
\subset
\Fonct(\Delta_0,\Fil^{h(k-1)n}\bD_k(\Zp)\otimes L)
$$
et son image dans $\Fonct(\Delta_0,\Gr^{h(k-1)n}\bD_k(\Zp))$
appartient à $\Hom_{\Gamma_0}(\Delta_0,\Gr^{h(k-1)n}\bD_k(\Zp))$.
\end{lem}

\begin{proof}
La première congruence vient de
l'hypothèse sur $\overline{\Phi}_{n-1}$
et de l'inclusion $\varphi^{-1} L \subset L$:
\begin{equation*}
\begin{split}
\overline{\Phi}_n = \cU_p^h( \widehat{\Phi}_{n-1})
&\equiv (1\otimes \varphi)^{-h}\widehat{\Phi}_{n-1} \bmod
p^{h(k-2)(n-1)}\Fonct(\Delta_0,\Fil^{h(n-1)}\bD_k(\Zp)\otimes L)
\\&
\equiv (1\otimes \varphi)^{-h}\overline{\Phi}_{n-1}\bmod p^{h(k-2)(n-1)}
\Fonct(\Delta_0,\Fil^{h(n-1)}\bD_k(\Zp)\otimes L).
\end{split}
\end{equation*}
\begin{equation*}
\end{equation*}
En utilisant que
$$\cU_p^h\Fonct(\Delta_0,\Fil^{h(n-1)}\bD_k(\Zp)\otimes L)\subset p^{h(k-2)}
\Fonct(\Delta_0,\Fil^{hn}\bD_k(\Zp)\otimes L),$$ on a
\begin{equation*}
\begin{split}
\cU_p(\overline{\Phi}_{n})&=
\cU_p\circ \cU_p^h (\widehat{\Phi}_{n-1})=
\cU_p^h\circ \cU_p (\widehat{\Phi}_{n-1})\\
&
\equiv (\cU_p^h\otimes \varphi^{-1})(\widehat{\Phi}_{n-1}) \bmod p^{h(k-2)n}
\Fonct(\Delta_0,\Fil^{hn}\bD_k(\Zp)\otimes L)
\\
&\equiv (1\otimes \varphi)^{-1} \overline{\Phi}_{n}
\bmod p^{h(k-2)n}\Fonct(\Delta_0,\Fil^{hn}\bD_k(\Zp)\otimes L).
\end{split}
\end{equation*}
Pour deux relèvements $F$ et $F'$ de $\overline{\Phi}_{n-1}$, on a
$$(F -F')(\delta) \in \Fil^{h(k-1)n} \bD_k(\Zp)\otimes L$$
pour tout $\delta \in \Delta_0$
et donc
\begin{equation*}
\begin{split}
\cU_p^h(F-F')\in \Fonct(\Delta_0, &p^{h(k-1)} \Fil^{h(k-1)(n-1)}\bD_k(\Zp))
\\ &\subset \Fonct(\Delta_0, \Fil^{h(k-1)n}\bD_k(\Zp))
\end{split}
\end{equation*}
par le corollaire \ref{crucial}. Par unicité, on montre
que $\overline{\Phi}_n$ est un homomorphisme
de $\Delta_0$ à valeurs dans $\Gr^{h(k-1)n}\bD_k(\Zp))$,
invariant par $\Gamma_0$, puis que
$$\cU_p(\overline{\Phi}_n) = (1\otimes \varphi)^{-1}\overline{\Phi}_n \in
\Hom_{\Gamma_0}(\Delta_0,\Gr^{h(k-1)n}\bD_k(\Zp)).$$
\end{proof}
\begin{lem}On reprend les notations des lemmes précédents.
Supposons que $\lambda < h(k-1)$.
La suite $\Phi_{n}=(1\otimes \varphi)^{hn} \overline{\Phi}_{n}$ converge dans
$\Hom(\Delta_0,\bD_k^\dag\otimes D)$. Sa limite
$\Phi_\infty$ est indépendante des choix faits, appartient à
$\Hom_{\Gamma_0}(\Delta_0,\bD_k^\dag\otimes D)$ et
vérifie $(\cU_p\otimes \varphi)(\Phi_\infty) = \Phi_\infty$.
\end{lem}
\begin{proof}
On déduit du lemme précédent et de ce que $\varphi^{-h} L \subset p^{-\lambda}L$
que
\begin{equation*}
\begin{split}
(1\otimes \varphi)^{h(n+1)} \overline{\Phi}_{n+1} \equiv
(1\otimes \varphi)^{hn} \overline{\Phi}_{n}
  \bmod p^{h(k-2)n-\lambda (n+1)} \Hom(\Delta_0,\Fil^{hn}\bD_k(\Zp)\otimes L).
\end{split}
\end{equation*}
Donc,
\begin{equation*}
\begin{split}
(1\otimes \varphi)^{h(n+1)} \overline{\Phi}_{n+1} \equiv
(1\otimes \varphi)^{hn} \overline{\Phi}_{n}
  \bmod \Hom(\Delta_0,\Fil^{t_n}\bD_k(\Zp)\otimes L)
\end{split}
\end{equation*}
avec $t_n=h(k-1)n-\lambda (n+1)=(h(k-1) -\lambda)n -\lambda$.
Pour $\lambda < h(k-1)$, la suite $\Phi_n$ converge donc
dans $\Hom(\Delta_0,\bD_k(\Zp)\otimes L)$,
sa limite est invariante par $\Gamma_0$ et vérifie
$$(\cU_p\otimes\varphi)(\Phi_\infty) = \Phi_\infty.$$
\end{proof}
La proposition \ref{prop:phioo} se déduit des lemmes précédents. Remarquons
que plus $\lambda$ est petit, plus la convergence est rapide.
\begin{prop}
Soient $\Phi_\infty$ un élément de
$\Hom_{\Gamma_0}(\Delta_0, \bD_k(\Zp)\otimes D)$ vérifiant
$(\cU_p\otimes \varphi)\Phi_\infty=\Phi_\infty$,
$\Phi$ son image dans
$\Hom_{\Gamma_0}(\Delta_0, \Qp[\XX]_{k-2}\otimes D)$ par $\rho_{k,*}$ et
$\mu_{\Phi}$ la restriction à $\Zp^\times$ de $\tilde{\mu}_{\Phi}=\Phi_\infty((\infty,0))$.
Alors, pour $n\geq 1$ et $a\in\Zp^\times$, on a
\begin{equation*}
\begin{split}
\displaystyle
\int_{a+p^n\Zp} f \,d\mu_{\Phi}
&= \varphi^n
   \bildist{\Phi_\infty\actk\smallmat{1&a\\0&p^n}((\infty,0))}{f},\\
\displaystyle
\int z^j \,d\mu_{\Phi}
&= (1- p^j\varphi)\bildist{\Phi_\infty((\infty,0))}{z^j}.
\end{split}
\end{equation*}
\end{prop}
\begin{proof}
Pour $a\in\Zp$, notons $F = \charun_{a+p^n\Zp} \cdot f$.
L'identité
$(\cU_p\otimes \varphi)(\Phi_\infty)=\Phi_\infty$
et le comportement de l'action de $\Gl_2(\QQ)$ impliquent que pour $\delta\in \Delta_0$,
\begin{equation*}
\begin{split}
\bildist{(\cU_p\otimes \varphi)^n(\Phi_\infty)(\delta)}{F}
&=
\varphi^n\sum_{b=0}^{p^n-1}
\bildist{\Phi_\infty\actk \smallmat{1&b\\0&p^n}(\delta)}{F}\\
&=
\varphi^n\sum_{b=0}^{p^n-1}
\bildist{\Phi_\infty(\smallmat{1&b\\0&p^n}\delta)\actk\smallmat{1&b\\0&p^n}}{F}
\\
&=
\varphi^n
\bildist{\sum_{b=0}^{p^n-1}\Phi_\infty(\smallmat{1&b\\0&p^n}\delta)}
  {\smallmat{1&b\\0&p^n}\cdot_k F}.
\end{split}
\end{equation*}
Or,
\begin{equation*}
\paren{\smallmat{1&b\\0&p^n} \cdot_k F}(z) =
F(b+p^n z)
=\begin{cases}
f(a+p^n z)\cdot \charun_{\Zp}(z) &\text{si $a\equiv b \bmod p^n$}
\\
0 &\text{ si $a\not \equiv b \bmod p^n$.}
\end{cases}
\end{equation*}
Donc,
\begin{equation*}
\begin{split}
\bildist{(\cU_p\otimes \varphi)^n(\Phi_\infty)(\delta)}{F}
=&
\varphi^n\left(\bildist{\Phi_\infty(\smallmat{1&a\\0&p^n}\delta)}{f(a+p^n z)}\right)
\\
&=
\varphi^n\left(\bildist{\Phi_\infty \actk \smallmat{1&a\\0&p^n}(\delta)}{f}\right).
\end{split}
\end{equation*}
Comme
$$\int_{a+p^n\Zp} f \,d\widetilde{\mu}_{\Phi}
=\bildist{(\cU_p\otimes \varphi)^n(\Phi_\infty)((\infty,0))}{F},$$
on a pour $a$ premier à $p$ et $n \geq 1$,
  $$\int_{a+p^n\Zp} f \,d\mu_{\Phi}=\int_{a+p^n\Zp} f \,d\widetilde{\mu}_{\Phi}
  =
  \varphi^n\left(\bildist{\Phi_\infty\actk \smallmat{1&a\\0&p^n}((\infty,0))}{f}\right).
  $$
et pour $a=0$ et $n=1$,
\begin{equation*}
\int_{p\Zp} f \,d\widetilde{\mu}_{\Phi}=
\varphi\left(\bildist{\Phi_\infty \actk \smallmat{1&0\\0&p}((\infty,0))}{f}\right),
\end{equation*}
d'où
\begin{equation*}
\int f \,d\mu_{\Phi}=
\bildist
  {\paren{\Phi_\infty- (1\otimes\varphi)\Phi_\infty \actk \smallmat{1&0\\0&p}}
    ((\infty,0))}{f}.
\end{equation*}
On en déduit la proposition en utilisant le fait que
  $\smallmat{1&0\\0&p}(\infty,0)=(\infty,0)$.
\end{proof}
\begin{rem}
Pour $0 \leq j \leq k-2$, $\bildist{\Phi_\infty(\delta)}{z^j}$ est égal à
$\bildist{\Phi(\delta)}{z^j}$ pour $\delta\in \Delta_0$.
\end{rem}
\section{Applications aux formes paraboliques}
\label{application}
Prenons $\Gamma=\Gamma_0(N)$. Soit $\Psi$ un élément de
$\Hom_\Gamma (\Delta_0,\Qp[\XX]_{k-2})$.
Sous certaines conditions, nous allons le relever en
un symbole à valeurs dans $\bD_k(\Zp)$.

\subsection{Cas ordinaire}
\label{ordinaire}On suppose que
$N$ est premier à $p$ et que $\Psi$ est
vecteur propre pour l'opérateur $T_p(N)=\cU_p + V_p$ avec $V_p=\smallmat{p&0\\0&1}$:
$$T_p(N)\Psi=a_p \Psi.$$
On a
\begin{equation*}
\begin{split}
&\cU_p(\Psi)=T_p(N)\Psi - \Psi \actk \begin{pmatrix}p&0\\0&1\end{pmatrix}
=a_p\Psi - \Psi \actk \begin{pmatrix}p&0\\0&1\end{pmatrix}
\\
&\cU_p\paren{\Psi \actk \begin{pmatrix}p&0\\0&1\end{pmatrix}}=
p \Psi \actk \begin{pmatrix}p&0\\0&p\end{pmatrix}= p^{k-1} \Psi,
 \end{split}
\end{equation*}
soit $(\cU_p^2 - a_p \cU_p + p^{k-1})\Psi = 0$.
Supposons que
$X^2-a_pX+p^{k-1}$ a une racine dans $\Zp$, i.e. $\ord_p(a_p) < (k-1)/2$,
par exemple
lorsque $p$ ne divise pas $a_p$. On note $\alpha$ une racine de valuation
$\lambda$ minimale; le produit des deux racines étant $p^{k-1}$, on a
$$ \lambda \leq \frac{k-1}{2} < k-1.$$
Alors,
$$\Phi=
\Psi -\frac{1}{\alpha}\Psi\actk \begin{pmatrix}p&0\\0&1\end{pmatrix}$$
appartient à $\Hom_{\Gamma_0} (\Delta_0,\Qp[\XX]_{k-2})$ et vérifie
$ \cU_p(\Phi)= \alpha \Phi$.
On peut lui appliquer la proposition \ref{prop:phioo} avec $h=1$,
$D=\Qp$ muni de la multiplication $\varphi$ par $\alpha^{-1}$, $L = \Zp$
et $\lambda = \ord_p(\alpha)$:
il existe un élément $\Phi_\infty\in \Hom_{\Gamma_0}(\Delta_0,\bD_k(\Zp))$
vérifiant
\begin{equation}
\label{eq:ordinaire}
\begin{cases}
\rho_{k,*}(\Phi_\infty) &= \Phi = \Psi -\frac{1}{\alpha}\Psi\actk \smallmat{p&0\\0&1}\\
\cU_p(\Phi_\infty) &=\alpha\Phi_\infty.
\end{cases}
\end{equation}

\subsection{Cas supersingulier}
\label{superss}
Plus généralement, en supposant toujours que $N$ est premier à $p$ et que
$\Psi$ est vecteur propre pour $T_p(N)$ de valeur propre $a_p$,
soit $D=\Qp e_1 \oplus \Qp e_2$ le $\Qp$-espace vectoriel de dimension 2 muni
de l'endomorphisme $\varphi$ donné dans la base $(e_1,e_2)$ par la matrice
de $M_2(\QQ)$ suivante:
$$ \begin{pmatrix}
0 & p^{1-k}\\
-1& p^{1-k}a_p
\end{pmatrix}.$$
Son inverse $\varphi^{-1}$ est de matrice $\smallmat{a_p&-1\\p^{k-1}&0}$
dans $M_2(\ZZ)$ et
on considère le réseau $L=\Zp e_1 \oplus \Zp e_2$, stable par $\varphi^{-1}$.
L'endomorphisme $\varphi$ vérifie $p^{k-1}\varphi^2 - a_p \varphi + \id=0$.
Alors,
$$\Phi=\Psi e_1 - \Psi\actk \smallmat{p&0\\0&1}\varphi e_1
=\Psi e_1 + \Psi\actk \smallmat{p&0\\0&1} e_2$$
appartient à
$\Hom_{\Gamma_0}(\Delta_0,\Qp[\XX]_{k-2}\otimes L)$
et vérifie
$(\cU_p\otimes \varphi)(\Phi) = \Phi$.
La matrice de $\varphi^{2}$ est
$$p^{2-2k}\smallmat{-p^{k-1}&a_p\\-p^{k-1}a_p&a_p^2-p^{k-1}},$$
ce qui montre que
$\varphi^{2}L \subset p^{-\lambda} L$ avec
$$\lambda=2k-2-\min(k-1, \ord_p(a_p)) =
\begin{cases}
k-1 & \text{si $\ord_p(a_p) \geq k-1$},\\
2k-2 - \ord_p(a_p) & \text{sinon}.\\
\end{cases}
$$
En particulier, $\lambda = k-1$ si $a_p = 0$.
On applique la proposition \ref{prop:phioo} avec $h=2$ et $\lambda$ ainsi
défini:
il existe donc un élément
$\Phi_\infty\in \Hom_{\Gamma_0}(\Delta_0,\bD_k(\Zp)\otimes D)$ tel que
\begin{equation}
\label{eq:superss}
\begin{cases}
\rho_{k,*}(\Phi_\infty) &= \Psi e_1 - \Psi\actk \smallmat{p&0\\0&1}\varphi^{-1} e_1\\
(\cU_p\otimes \varphi)(\Phi_\infty) &= \Phi_\infty.
\end{cases}
\end{equation}

\subsection{Cas semistable}
\label{semistable}On suppose maintenant que
$N$ est exactement divisible par $p$, $\Psi$ étant toujours vecteur propre
pour l'opérateur $T_p(N)=\cU_p$. Dans ce cas, la valeur propre
$a_p$ est égale à $\pm 1$. On peut alors appliquer la proposition
\ref{prop:phioo}
à $\Phi=\Psi$, $h=1$ et $D=\Qp$ muni de la multiplication par $a_p$:
il existe un élément $\Phi_\infty\in \Hom_{\Gamma_0}(\Delta_0,\bD_k(\Zp))$
vérifiant
\begin{equation}
\label{eq:semistable}
\begin{cases}
\rho_{k,*}(\Phi_\infty) &= \Psi\\
\cU_p(\Phi_\infty) &=a_p\Phi_\infty.
\end{cases}
\end{equation}

\subsection{\texorpdfstring{Fonctions $L$ $p$-adiques}{Lg}}
Soit $F$ une forme parabolique pour $\Gamma_0(N)$ et $\Per(F)$ le symbole
\footnote{Dans \cite{periodes}, il est à valeurs dans l'espace vectoriel
$\QQ[x,y]_{k-2}$ isomorphe à $\QQ[\XX]_{k-2}$ par $P\mapsto P(\XX,1)$
de réciproque $P \mapsto y^{k-2}P(x/y)$.
}
à valeurs dans $\CC[\XX]_{k-2}$ qui lui est associé:
on a donc pour tout rationnel $r$
$$\Per(F)((\infty,r))=\int_{i\infty}^r F(t) (\XX t+1)^{k-2}dt
=
\sum_{j=0}^{k-2}\binom{k-2}{j} \paren{\int_{i\infty}^r F(t) t^jdt} \XX^j.
$$
Il existe un $\QQ$-sous-espace vectoriel
$\Omega_F$ de $\CC$ de dimension 2 et un réseau $\cL_F$ de $\Omega_F$
tel que $\Per(F)(\Delta_0)$ soit contenu dans
$\cL_F \otimes \ZZ[\XX]_k$. Nous avons pris le parti
de ne pas couper selon les parties $+$ et $-$, ce qui signifie de travailler
peut-être de manière osée dans un $\QQ$-espace vectoriel de dimension finie
qui n'est pas une droite. Nous laissons le lecteur faire les projections nécessaires
par le choix de bases.

Plaçons-nous dans une des situations du paragraphe \ref{application}
pour $\Per(F)$ et notons
$\Perp(F)$ le symbole associé
$$\Perp(F)\in \Hom_{\Gamma_0}(\Delta_0, \cL_F\otimes \Qp[\XX]_{k-2}\otimes D ).$$
Soit $\Perpi(F)\in \Hom_{\Gamma_0}(\Delta_0,\cL_F\otimes \bD_k(\Zp)\otimes D)$
le symbole tel que
\begin{equation*}
\begin{split}
\rho_{k,*}(\Perpi(F))&=\Perp(F)
\\
(\cU_p\otimes \varphi)(\Perpi(F))&=\Perpi(F).
\end{split}
\end{equation*}
Soit $\tilde{\mu}_{F,p}$ la distribution associée à valeurs dans $\cL_F\otimes D$:
$\tilde{\mu}_{F,p}=\Perpi(F)((\infty,0))$
et $\mu_{F,p}$ sa restriction à $\Zp$.
On a
\begin{equation*}
\int_{a+p^n\Zp} f \,d\tilde{\mu}_{F,p}
=\varphi^{n}\bildist{\Perpi(F)\actk \betab{a}{p^n}((\infty,0))}{f}.
\end{equation*}
Prenons $f(z)=(\XX z +1)^{k-2}$.
On a en utilisant le fait que
$\Per(F\actk \gamma)=\Per(F)\actk \gamma$,
\begin{equation*}
\begin{split}
\bildist{\Per(F)\actk \betab{a}{p^n}((\infty,0))}{f}
&=
\bildist{\Per(F\actk \betab{a}{p^n})((\infty,0))}{f}
\\
\bildist{\Per(F)\actk \smallmat{p&0\\0&1}\betab{a}{p^n}((\infty,0))}{f}
&=p^{k-2}
\bildist{\Per(F\actk \betab{a}{p^{n-1}})((\infty,0))}{f}.
\end{split}
\end{equation*}
Si $g$ est une fonction périodique de $\ZZ$ dans $\CC$, on pose
$L(F,s,g)=\sum_{n=1}^\infty g(n) a_n n^s$
où $\sum_{n\geq 1} a_n q^n$ est le $q$-développement
de $F$. Pour $g=1$, on obtient la fonction $L$ complexe $L(F,s)$ de $F$.
On note $L_{\{p\}}$ la fonction $L$ incomplète en $p$:
$$L_{\{p\}}(F,s)=(1-a_p p^{-s} + p^{k-1-2s})L(F,s).$$
On a
\begin{equation*}
\begin{split}
\bildist{\Per(F\actk \betab{a}{p^{n}})((\infty,0))}{(\XX z +1)^{k-2}}
&=p^{-n}\int_{i\infty}^0 F(\frac{z+a}{p^n}) (\XX z +1)^{k-2} dz
\\&=
\int_{i\infty}^0 F(z+\frac{a}{p^n}) (p^n\XX z +1)^{k-2} dz
\\
&= \sum_{j=0}^{k-2}\binom{k-2}{j}p^{nj}
\frac{\Gamma(j+1)}{(-2i\pi)^{j+1}}L(F,j+1,\epsilona{a}{p^n})X^j
\end{split}
\end{equation*}
où $\epsilona{a}{p^n}(x)=\exp(\frac{2i\pi a x}{p^n})$
est périodique sur $\ZZ$ de période divisant $p^n$.
Lorsque $a$ est premier à $p$ et que $n$ est supérieur ou égal à $1$, on a
\begin{equation*}
\begin{split}
\int_{a+p^n\Zp} z^j \,d\mu_{F,p}&=
\frac{\Gamma(j+1)}{(-2i\pi)^{j+1}}p^{nj} \paren{
  L(F,j+1,\epsilona{a}{p^{n}})\varphi^{n}e_1
- p^{k-2-j}L(F,j+1,\epsilona{a}{p^{n-1}})\varphi^{n+1}e_1
 }\\
&=
\frac{\Gamma(j+1)}{(-2i\pi)^{j+1}}L\paren{F,j+1, H}
\end{split}
\end{equation*}
avec une notation un peu osée
\begin{equation*}
H=\paren{\epsilona{a}{p^{n}} - p^{k-2-j}\epsilona{a}{p^{n-1}}\varphi}
   p^{nj}\varphi^{n}e_1.
\end{equation*}
On en déduit que si $\chi$ est un caractère de conducteur $p^n$ avec $n>0$, on a
\begin{equation*}
(p^{-j-1}\varphi)^{-n}\int_{\Zp} \chi(z) z^j \,d\mu_{F,p}=
\frac{\Gamma(j+1)}{(-2i\pi)^{j+1}}\frac{L_{\{p\}}(F,j+1,\overline{\chi})}{G(\overline{\chi})}e_1.
\end{equation*}
Pour $n=0$, on a de même
\begin{equation*}
\int_{\Zp} z^j \,d\widetilde{\mu}_{F,p}=
\frac{\Gamma(j+1)}{(-2i\pi)^{j+1}}
  L(F,j+1)(1- p^{k-2-j} \varphi)e_1,
\end{equation*}
d'où,
\begin{equation*}
\int_{\Zp} z^j d\mu_{F,p}=
\frac{\Gamma(j+1)}{(-2i\pi)^{j+1}}L(F,j+1)(1-p^j \varphi)(1- p^{k-2-j} \varphi)e_1
\end{equation*}
La relation
$\varphi^{-2}-a_p\varphi^{-1} + p^{k-1}\id=0$ implique que pour $0\leq j\leq k-2$,
$$(1-p^{k-2-j} \varphi)(1-p^{-j-1} \varphi^{-1})=
1-a_p p^{-(j+1)} + p^{k-1-2(j+1)}.$$
On en déduit que
\begin{equation*}
(1-p^{j}\varphi)^{-1}(1-p^{-j-1}\varphi^{-1})
\int_{\Zp} z^j d\mu_{F,p}=
\frac{\Gamma(j+1)}{(-2i\pi)^{j+1}}L_{\{p\}}(F,j+1)e_1.
\end{equation*}
Dans le cas ordinaire, $\varphi$ est simplement la multiplication
par $\alpha^{-1}$. D'où (en prenant $e_1=1$)
\begin{equation*}
(1-p^{j}\alpha^{-1})^{-1}(1-p^{-j-1}\alpha)
\int_{\Zp} z^j d\mu_{F,p}=
\frac{\Gamma(j+1)}{(-2i\pi)^{j+1}}L_{\{p\}}(F,j+1)
\end{equation*}
et pour $\chi$ un caractère de conducteur $p^n$ avec $n>0$,
\begin{equation*}
(p^{-j-1}\alpha)^{-n}
\int_{\Zp} \chi(z) z^j d\mu_{F,p}=
\frac{\Gamma(j+1)}{(-2i\pi)^{j+1}}\frac{L_{\{p\}}(F,j+1,\overline{\chi})}{G(\overline{\chi})}
.
\end{equation*}
Il y a plusieurs manières de définir la fonction $L$ $p$-adique
associée à une distribution. Si on la voit comme une fonction sur
les caractères $p$-adiques de $\Zp^*$,
la fonction $L$ $p$-adique associée à $F$ est alors (à des normalisations près)
$$L_p(\chi)=\int_{\Zp^*} \chi(z) d\mu_{F,p}(z).$$

On retrouve ainsi les formules usuelles reliant les fonctions $L$ $p$-adiques
aux fonctions $L$ complexes
(dans le cas supersingulier, voir par exemple \cite{experimenta})
une fois choisies les bases usuelles de $\Omega_F$ quand $F$ est vecteur propre
de tous les opérateurs de Hecke.

\bibliographystyle{smfplain}
\selectbiblanguage{french}
\bibliography{biblio}

\providecommand{\bysame}{\leavevmode ---\ }
\providecommand{\og}{``}
\providecommand{\fg}{''}
\providecommand{\smfandname}{\&}
\providecommand{\smfedsname}{\'eds.}
\providecommand{\smfedname}{\'ed.}
\providecommand{\smfmastersthesisname}{M\'emoire}
\providecommand{\smfphdthesisname}{Th\`ese}
\begin{thebibliography}{1}

\bibitem{av}
{\scshape Y.~Amice {\normalfont \smfandname} J.~Vélu} -- {\og Distributions
  $p$-adiques associées aux séries de {H}ecke\fg}, \emph{Astérisque, Soc. Math.
  France, Paris} \textbf{24-25} (1975), p.~119--131.

\bibitem{periodes}
{\scshape D.~Bernardi {\normalfont \smfandname} B.~Perrin-Riou} -- {\og
  Symboles modulaires et produit de {P}etersson\fg}, \emph{{Journal de Théorie
  des Nombres de Bordeaux}} \textbf{32} (2020), p.~795--859.

\bibitem{mgreenberg}
{\scshape M.~Greenberg} -- {\og Lifting modular symbols of non-critical
  slope\fg}, \emph{Israel Journal of Mathematics} \textbf{161} (2007),
  p.~141--155.

\bibitem{MTT}
{\scshape B.~Mazur, J.~Tate {\normalfont \smfandname} J.~Teitelbaum} -- {\og On
  p-adic analogues of the conjectures of {B}irch and {S}winnerton-dyer\fg},
  \emph{Inventiones mathematicae} \textbf{84} (1986), p.~1--48.

\bibitem{panchishkin02}
{\scshape A.~Panchishkin} -- {\og {A new method of constructing $p$-adic
  $L$-functions associated with modular forms}\fg}, \emph{Mosc. Math. J.}
  \textbf{2} (2002), p.~313--328.

\bibitem{experimenta}
{\scshape B.~Perrin-Riou} -- {\og Arithmétique des courbes elliptiques à
  réduction supersingulière en $p$\fg}, \emph{Experimental Mathematics}
  \textbf{12} (2003), p.~155--186.

\bibitem{PS}
{\scshape R.~Pollack {\normalfont \smfandname} G.~Stevens} -- {\og
  Overconvergent modular symbols and $p$-adic {L}-functions\fg}, \emph{Annales
  scientifiques de l'ENS} \textbf{44} (2011), p.~1--42.

\bibitem{pari}
{\scshape {The PARI Group}} -- {\og {PARI}/{GP} version \texttt{2.13.0}\fg},
  \url{http://pari.math.u-bordeaux.fr}, 2020.

\end{thebibliography}
\end{document}